\documentclass[12pt,a4paper]{amsart}
\usepackage[a4paper]{geometry}
\usepackage{mathptmx} % Times font
\DeclareMathAlphabet{\mathcal}{OMS}{cmsy}{m}{n} % but do not change mathcal symbols

\usepackage{mathrsfs,alphalph,perpage}

\newtheorem{theorem}{Theorem}
\newtheorem{mytheorem}{Theorem}
\theoremstyle{definition}
\newtheorem{example}{Example}

\makeatletter
\newalphalph{\fnsymbolwrap}[wrap]{\@fnsymbol}{}
\makeatother

\MakePerPage[2]{footnote}

\begin{document}

\title{On the last question of Stefan Banach}
\author{Pasha Zusmanovich}
\address{
Department of Mathematics, University of Ostrava, Ostrava, Czech Republic
}
\email{pasha.zusmanovich@gmail.com}
\date{last minor revision June 2, 2018}
\thanks{arXiv:1408.2982; Expos. Math. \textbf{34} (2016), N4, 454--466}


\begin{abstract}
We discuss the last question of Banach, posed by him in 1944, shortly before 
his death, about extension of a ternary map to superposition of a binary 
map. We try to put things into the context of Polish mathematics of that time,
and touch upon subsequent developments in such diverse areas as multivalued 
logics, binary and ternary semigroups, theory of clones, and Hilbert's 13th 
problem. Making almost a full circle in time, we show how variants of 
Banach's question may be settled using a 1949 idea of Jacobson about envelopes 
of Lie triple systems.
\end{abstract}

\maketitle

\section*{Introduction}

At the end of 1944, shortly before his death in August 1945, Stefan Banach 
regularly met with Andrzej Alexiewicz, then a fresh PhD from the (underground
at that time) Lvov University. During these meetings, they discussed a lot of 
mathematics, and Alexiewicz kept a diary whose mathematical part, including 
questions posed by Banach, was published posthumously by Alexiewicz's son in
\cite{alex}. The last entry in this diary, dated December 29, 1944, reads:

\medskip

\begin{center}
\parbox{15cm}{
``There exists a nontrivial example of ternary multiplication, which is not 
generated by binary multiplication (Banach). Can any finite set with ternary
commutative multiplication be extended so that ternary multiplication is
generated by a binary multiplication?''
}
\end{center}

\medskip

In what follows, the claim and the question from this passage will be referred 
as ``Banach's claim'' and ``Banach's question''. This seems to be the last 
question of Banach which appeared in the ``literature'' 
(broadly interpreted)\footnote{
Four Banach's papers were published posthumously in 1948, but they were prepared
from manuscripts apparently dated around 1940 (cf. \cite{banach-oeuvr}).
}.
It is interesting that among all other questions posed by Banach -- in his 
papers, in the Scottish Book \cite{scott}, in other diary entries of 
\cite{alex}, in the problem sections of \emph{Fundamenta Mathematicae} and 
\emph{Colloquium Mathematicum} (the latter being entered posthumously) --
this is the only one which does not belong to the field of analysis. (The only 
possible exception is Question 47 from \cite{scott} about permutations of 
infinite matrices, which, incidentally, also deals with (im)possibility of 
building a certain class of maps from ``simpler'' ones).
On the other hand, that Banach was interested in such sort of questions -- 
belonging, somewhat vaguely, to a crossroad of universal algebra, discrete 
mathematics, logic\footnote{
According to Hugo Steinhaus (cf. \cite{steinh}), Banach 
``did not relish any logic research although he understood it perfectly''. 
On the other hand, in \cite[p.~40]{murawski} several instances of Banach's 
involvement into contemporary logical activity are given. 
Anyhow, ``logic'' is a vast field, and, as we try to argue below, there are 
certain connections between Banach's question and some logical investigations 
cultivated in Poland (especially Warsaw) between the two world wars.}, 
and combinatorics -- is, perhaps, not accidental at all, as 
such a crossroad, along the functional analysis, the main Banach's occupation, 
was another ``Polish speciality'' at that time (and long thereafter).

It is the purpose of this note to discuss (and answer) possible interpretations
of this question. For the historical context we refer the reader to a variety of
excellent sources: to \cite{banach} for the last years of Banach in Lvov during
two Russian and one German occupations; to Duda \cite{duda}, 
Turowicz \cite{turowicz}, and to introductory chapters of \cite{scott} for the 
general atmosphere of mathematical Lvov; and to Murawski \cite{murawski} and
Wole\'nski \cite{wolenski}, \cite{wolenski-lvov} for a unique interaction of 
mathematics and logic in inter-war Poland.

\section{A 1955 paper by {\L}o\'s, Hilbert's 13th problem, and functional completeness}\label{sec-los}

There is at least another mention of a variant of Banach's question in the 
literature, namely by Jerzy {\L}o\'s, a student at the Lvov University during 
1937--1939, in \cite[\S 16]{los}. There, also with reference to Alexiewicz, 
{\L}o\'s writes:

\medskip

\begin{center}
\parbox{15cm}{
``During the last war Banach showed
that not every ternary semigroup is reducible and he put forward the problem 
whether every ternary semigroup may be extended to a reducible one''.
}
\end{center}

\medskip

Here by a ternary semigroup one means a set $X$ with a ternary map 
\begin{equation}\label{eq-f}
f: X \times X \times X \to X ,
\end{equation}
subject to a ternary variant of associativity:
$$
f(f(x,y,z),u,v) = f(x,f(y,z,u),v) = f(x,y,f(z,u,v))
$$
for any $x,y,z,u,v \in X$, and by a reducible ternary semigroup one means a 
ternary semigroup $(X,f)$ with multiplication given by 
\begin{equation}\label{eq-star}
f(x,y,z) = (x*y)*z ,
\end{equation}
where $(X,*)$ is an ordinary (i.e., binary) semigroup structure on the same
underlying set $X$. {\L}o\'s answers the question in affirmative as a consequence of his general 
results about extensibility of first-order logical models. Later an alternative and more constructive proof was given by Monk and Sioson 
\cite[Theorems 1 and 2]{monk-sioson}. The latter paper contains also examples of
ternary (in fact, $n$-ary for any positive integer $n$) semigroups not 
representable in the form (\ref{eq-star}), thus validating this variant of 
Banach's claim.

However, this variant of Banach's question is more narrow in scope than those 
presented in \cite{alex}. Can the latter question be interpreted in a different 
way? In the absence of additional qualifications, the most general reading is 
the following: ``multiplication''  on a set $X$ means an arbitrary ternary map
(\ref{eq-f}), without associativity, or any other, for that matter, constraint.
But what does ``generation'' mean?

Virtually all Banach's works are devoted to (proper) analysis and/or measure 
theory (including the famous Banach--Tarski paradox), with, perhaps, only two 
exceptions: a joint unpublished work with Stanis{\l}aw Mazur on computable 
analysis (a constructive approach to analysis)\footnote{
That it remained unpublished is probably not surprising, owing to Mazur's 
(non)publishing habits, even of joint papers, cf., e.g., \cite[p.~14]{turowicz}.
}, and a short paper \cite{banach-paper}, which touches, seemingly, a similar
question. There, Banach gives a shorter proof of an earlier result of 
Wac{\l}aw Sierpi\'nski \cite{sierp-banach}:

\begin{theorem}[Sierpi\'nski]\label{th-un}
For any countable set $\mathcal S$ of unary maps $X \to X$ on an infinite set 
$X$, there are two unary maps $f,g: X \to X$ (not necessary belonging to 
$\mathcal S$), such that any map from $\mathcal S$ is a composition of $f$ and
$g$.
\end{theorem}

This statement is, essentially, about $2$-generation of a countable
transformation semigroup, and was apparently rediscovered in the literature
several times afterwards. It admits various generalizations to other 
semigroups and groups, some of them, especially in topological setting (where
``generation'' is understood up to the closure), were pursued by Sierpi\'nski
and two Banach's students, Stanis{\l}aw Ulam (PhD from the Lvov Polytechnic, 
1933), and J\'ozef Schreier (PhD from the Lvov University, 1934; cf. 
Guran and Prytula \cite{prytula} for an interesting discussion, and the 
bestseller \cite[p.~82]{ulam} for Ulam's account how his joint work with 
Schreier secured him a place in the Harvard Society of Fellows in 1936). 

Yet Theorem \ref{th-un} about unary maps proved by Sierpi\'nski and Banach 
admits generalizations in another direction -- to the multiary maps. (Note the 
drastic difference between superposition of unary maps, which reduces to a mere
composition and hence is associative, and the general case of superposition of
multiary maps for which associativity is not even well defined). This may 
suggest that ``generation'' in Banach's question can be interpreted as an 
arbitrary superposition of maps (with possible repetition of variables). For 
example, among the ternary maps obtained by superposition of a binary map $f$ 
with itself, are the maps
$$
f(f(x,y),z), \quad f(f(x,y),f(y,z)), \quad f(x,f(y,f(z,z))), \quad
f(f(f(x,y),f(x,z)),x),
$$
etc.

In this general setting, however, Banach's claim becomes false: any map of an 
arbitrary (finite) arity on any set -- in fact, any countable set of such
maps -- can be generated in this fashion by one binary map (a multiary analog
of Theorem \ref{th-un}). More precisely:

\begin{theorem}[Webb, {\L}o\'s]\label{th-webb-los}
For any countable set $\mathcal S$ of multiary maps 
$X \times \dots \times X \to X$ (of, possibly, different arities) on a set $X$,
there is a binary map $f: X \times X \to X$ (again, not necessary belonging to 
$\mathcal S$) such that any map from $\mathcal S$ is a superposition of $f$.
\end{theorem}

This was proved several times in the literature: for the first time, in 
\cite{webb} for the case of $X$ finite (the required binary map is a multiary 
generalization of the Sheffer stroke). In the case of $X$ infinite, a different,
and much shorter\footnote{
A quote from Erd\"os \cite{erdos}: ``Now it frequently happens in problems of 
this sort that the infinite dimensional case is easier to settle than the finite
dimensional analogues. This moved Ulam and me to paraphrase a well known maxim
of the American armed forces in WWII: 'The difficult we do immediately, the 
impossible takes a little longer', viz: 'The infinite we do immediately, the 
finite takes a little longer' ''.
}
proof was presented, again, by {\L}o\'s \cite{los-superp} (a similar, and yet
simpler, proof was rediscovered more than half century later by Goldstern 
\cite{gold}). These proofs, in its turn, are based on an another result of 
Sierpi\'nski \cite{sierp} to the effect that any single multiary map on an 
infinite set can be obtained as a superposition of (possibly, several) binary 
maps.

The latter paper of Sierpi\'nski appeared around the same time as Banach's 
question under discussion: the same 1945 issue of \emph{Fundamenta Mathematicae}
in which the paper was published, the first one after the 6-year break occurred
during WWII, contains an announcement about Banach's death. The same paper 
contains also another elementary, but interesting for us result: 

\begin{theorem}[P\'olya--Szeg\"o]\label{th-tri}
For any binary bijection $g: X \times X \to X$ on a (necessary infinite) set 
$X$, and any ternary map $f: X \times X \times X \to X$, there is another 
binary map $h: X \times X \to X$ such that
\begin{equation*}
f(x,y,z) = g(h(x,y),z)
\end{equation*}
for any $x,y,z\in X$.
\end{theorem}

(The statement readily generalizes to $n$-ary maps $f$). The result was, 
however, not new at that time: it appeared as the solution to the problem 119, 
imaginatively entitled ``Are there actually functions of $3$ variables?'', of 
Part 2 in the first 1925 edition of the famous book \cite{polya-szego}. 
Sierpi\'nski's interest in this topic stems, evidently, from (a variant of) 
Hilbert's 13th problem: does there exist a real continuous function in $3$ 
variables which cannot be represented as a superposition of real continuous 
functions in $2$ variables? See, for example, his earlier paper 
\cite{sierp-hilb13} where Theorem \ref{th-tri} is proved for the case where $X$
is the set of real numbers (albeit without using axiom of choice which is 
necessary in the general case), with reference to the (in)famously, at that time, erroneous paper by Bieberbach about Hilbert's 13th 
problem.

Ulam, along with Mark Kac (PhD from the Lvov University, 1937, with Banach as a
member of examining committee), also an active participant of the Lvov 
mathematical scene until the end of 1930s\footnote{
Among the two, Ulam was more actively engaged in the Lvov mathematical life, and
not only because he was a few years Kac's senior, but by a more prosaic reason.
Speaks Kac (cf. \cite{feigenbaum}): ``I was less of a habitue of the Scottish 
Caf\'e... I was financially somewhat less affluent than Stan -- I was ... 
independently poor. And it did cost a little to visit in the Caf\'e''. 
}
(cf. \cite{scott}, \cite{ulam}, and \cite{feigenbaum}), has at least a cursory 
interest in Hilbert's 13th problem as well. In the collection of Ulam's 
problems\footnote{
Dedicated to Schreier's memory, who committed suicide to avoid being captured
by Germans during WWII.}
\cite[Chapter IV, Problem 2 and Chapter VI, Problem 5]{ulam-problems}, which he
positions as a sort of successor to the Scottish Book, as well in the joint 
Kac's and Ulam's popular book \cite[p.~163]{kac-ulam}, they note a famous
1956--1957 negative solution of the Hilbert's 13th problem by Kolmogorov and 
Arnold:

\begin{theorem}[\protect{Kolmogorov, Arnold, cf. \cite[Theorem 1]{arnold}}]\label{th-kolm-arn}
For any real continuous function in $3$ variables 
$f: [0,1] \times [0,1] \times [0,1] \to \mathbb R$, there are $9$ real 
continuous functions in $2$ variables $\varphi_{ij}, \psi_{ij}$, $i,j=1,2,3$, 
such that
$$
f(x,y,z) = \sum_{i=1}^3 \sum_{j=1}^3 \varphi_{ij}(\psi_{ij}(x,y),z)
$$
for any $0 \le x,y,z \le 1$.
\end{theorem}

Kac and Ulam proceed then by formulating various extensions of the problem (for
other classes of functions, for functions acting on more complicated spaces, 
etc.). In \cite{ulam-problems}, Ulam acknowledges Banach, among others, for ``the pleasure of past collaboration'', but it is unclear whether this interest
in the circle of questions related to Hilbert's 13th problem goes back to the 
Lvov years, and if yes, whether it has something to do with Banach's question. 
At any rate, as was noted on several occasions, ``the iteration and 
composition of functions'' was one of the main leitmotifs of Ulam's work 
during his life (cf., e.g., \cite[p.~x]{ulam-analog}). Of course, there is a 
very long way from Theorems \ref{th-un}--\ref{th-tri} to 
Theorem \ref{th-kolm-arn}, but they represent development along the same circle 
of ideas (Arnold himself used to cite Theorem \ref{th-tri} as one of the early precursors in his popular expositions of the history of
Hilbert's 13th problem).

\medskip

The questions whether that or another set of maps generates all maps within a 
given class, framed in terms of functional completeness of various multivalued 
propositional calculi (``multivalued logics'', or ``logistics'', as it was 
called then) and the corresponding truth tables, were also popular among Polish 
logicians at that time. In particular, Jan {\L}ukasiewicz was concerned about 
functional completeness of his famous $3$-valued logic -- that is, 
``implication'', the binary map \textit{\L} on the $3$-element set 
$\{0, \frac 12, 1\}$ given by ``multiplication table''
$$\begin{array}{c|ccc}
         & 0        & \frac 12 & 1 \\
\hline
0        & 1        & 1        & 1 \\
\frac 12 & \frac 12 & 1        & 1 \\
1        & 0        & \frac 12 & 1
\end{array}$$
together with ``negation'', the unary map $N$ defined by 
$$
0 \mapsto 1, \> {\textstyle \frac 12} \mapsto {\textstyle \frac 12}, \> 
1 \mapsto 0 ,
$$
and similar systems. (While {\L}ukasiewicz puts emphasis on the philosophical 
significance of many-valued logics -- cf., e.g., \cite[pp.~75--76]{murawski} --
most of the questions related to these logics, including question of functional 
completeness, naturally have a purely formal, mathematical, character). His
student Jerzy S{\l}upecki proved in his PhD thesis (cf. \cite{slup}) functional
incompleteness of the set $\{\text{\textit{\L}}, N\}$, and, in the positive 
direction, established functional completeness of the set 
$\{\text{\textit{\L}}, N, T\}$, where $T$ is the unary map sending all $3$ 
elements to $\frac 12$:

\begin{theorem}[S{\l}upecki]
Let $X = \{0, \frac 12, 1\}$ be a set of $3$ elements.
\begin{enumerate}
\item
There are multiary maps $X \times \dots \times X \to X$ which cannot be 
represented as a superposition of the maps {\L} and $N$.
\item
Any multiary map $X \times \dots \times X \to X$ is a superposition of the maps
{\L}, $N$, and $T$.
\end{enumerate}
\end{theorem}

(It is interesting to note that the same 1936--1937 issue of 
\emph{Comptes Rendus des S\'eances de la Soci\'et\'e des Sciences et des Lettres
de Varsovie} in which S{\l}upecki's paper was published, contains the paper 
\cite{webb-vars} of Donald L. Webb, a fresh PhD from Caltech and the author of 
the already mentioned \cite{webb}, in which he extends the results of the latter
paper. This, among other things, suggests that developments in this branch of logical calculi on both sides of Atlantic did not proceed in 
isolation at that time\footnote{
The personal contacts started, probably, with the visit of 
Willard Van Orman Quine (then at Harvard) to Warsaw in 1932. Ernest Nagel (then
at Columbia) made public (cf. \cite{nagel}) his interesting impressions of 
visiting Poland in 1935, but, written from the philosopher's rather than 
logician's standpoint (``the logical researches both at Warsaw and Lw\'ow are 
extraordinary specialized and technical''), these impressions, probably, 
contributed little to interchange of logical and mathematical ideas between the countries. Ulam was moving back and 
forth between US (Princeton and then Harvard) and Lvov in 1935--1939, but his 
interests, at least at that time, were outside logic. Webb's thesis advisor, 
Eric Temple Bell, an enthusiastic, albeit not always precise, writer of popular mathematical books,
praised {\L}ukasiewicz's $3$-valued logic as one of the four major contributions
``on the nature of truth'' during the last 6000 years 
(cf. \cite[pp.~258--262]{bell}), so it is, perhaps, not accidental that Webb,
a young Californian, published in a relatively obscure Warsaw journal.}). 
Earlier, another {\L}ukasiewicz student, Mordchaj Wajsberg has reported at the 
{\L}ukasiewicz--Tarski seminar\footnote{
``Professor Lukasiewicz's seminar at Warsaw was crowded with competent young 
men, incomparably better equipped in logic than students of like age in America,
who were expected to write as seminar exercises papers which elsewhere would be 
thought important enough for publication'', reports Nagel.
} 
some related results -- see a brief description at the very end (``Anmerkung'')
of \cite{wajsberg} (cf. also Surma \cite[\S ix]{surma}). In the same paper, more
binary maps generating all maps of arbitrary arity on a finite set are presented
without proof. Some of the Wajsberg's results apparently go back as early as 
1927, i.e. almost a decade before Webb.

More logical calculi, both complete and incomplete, were developed around this
time, first of all, by Emil Leon Post (born in Poland, otherwise not related to
that country\footnote{
A quote from \cite{wolenski}: ``When Tarski met Emil Post for the first time 
(in 1939 or 1940) he told him: 'You are the only logician who achieved something
important in propositional calculus without having anything to do with Poland'. 
Post answered: 'Oh, no, I was born in Bia{\l}ystok' ''.
}), as well as by {\L}ukasiewicz, S{\l}upecki, Boles{\l}aw Soboci\'nski 
(another PhD student of {\L}ukasiewicz), Eustachy \.{Z}yli\'nski (professor of
the Lvov University in 1919--1941), Zygmunt Zawirski, and others. Post's purely
formal treatment of many-valued logics (``a combinatorial scheme'' in the words
of Mostowski \cite[p.~3]{most}) is, perhaps, closer in spirit to mathematical
questions considered here than the philosophical attitude of {\L}ukasiewicz.

Another feature of the work of {\L}ukasiewicz's school and Polish logicians in 
general, was the constant quest for minimal, as far as possible, systems of 
axioms for that or another logical calculus. Minimality was understood both in 
terms of the number of axioms and their length (cf., e.g., 
\cite[pp.~390--391]{wolenski} and \cite[\S v]{surma})\footnote{
This fascination with minimal systems of logical axioms was not shared by 
everyone in Poland. An anti-utopian novel ``Nienasycenie'' by Stanis{\l}aw 
Ignacy Witkiewicz (a close friend of Leon Chwistek, professor of logic at the 
Lvov University, as well as of Alfred Tarski), written in 1927 and depicting 
conquest of Poland by enemy forces and establishment there of a totalitarian regime by the end of XX century, features 
a grotesque figure of logician Afanasol Benz (a Jew, 
stresses Witkiewicz) who invented a single axiom that nobody but him could 
understand, and from which all mathematics follows by a mere formal combination of symbols (cf. \cite[p.~93]{witkiewicz}).
}.
Banach's claim and question resonate well with this line of thought, as they can
be phrased as follows: if a certain $3$-term operation in an $n$-valued logic is
not generated by a $2$-term one (alas, as we have seen, this is wrong), the next
best thing to ask is to extend the $3$-term operation so that it will be 
generated that way.

\medskip

Of course, many further results about Hilbert's 13th problem, from one side, and
functional completeness and incompleteness of various logical calculi, from the
other, were obtained afterwards (cf., e.g., the survey by Vitushkin 
\cite{vitushkin} for the former, and the book by Lau \cite{lau} for the latter),
but their discussion will bring us far away from our main topic.

\section{Clones of analytic and ordered maps, counting superpositions, examples
to Banach's claim}

The ``generation'' in the results above is understood in the sense of the theory
of clones, i.e. when forming superposition of maps, repeated variables (and,
hence, superpositions of arbitrary length) are allowed. (Recall that clone
is a set of a multiary maps on a fixed set $X$ which contains projections 
on each variable, and is closed with respect to superposition; cf., e.g., 
\cite[\S 1.3]{lau}). Thus, Banach's claim can not be interpreted in terms of the
clone of all maps on the underlying set $X$, as Theorem \ref{th-webb-los} tells
that all the clone, including its ternary fragment, is generated by its binary fragment. It should be noted that 
if we assume that $X$ possess some additional -- topological, analytic, order, 
etc. -- structure, and consider not the clone of all maps, but the clone of maps
on $X$ preserving this structure -- then Theorem \ref{th-webb-los} about 
generation of all the maps by binary ones is no longer true. Some sporadic 
examples of various degree of sophistication:
\begin{enumerate}
\item
Not every real analytic function in $3$ 
variables can be represented as a superposition of analytic functions in two 
variables (cf. \cite[\S 3]{vitushkin}) -- a statement made already by Hilbert
in his original formulation of the 13th problem, whose proof utilizes some 
counting similar to elementary counting in the case of $X$ finite, and also
in the case of polynomials over a finite field in Example \ref{ex-gf},
see below.

\item
Superpositions of smooth functions in $2$ variables satisfy various differential
equations, not satisfied by all smooth functions in $3$ variables (cf. solution of Problem
119a, Part 2 of \cite{polya-szego} for the relevant calculations).

\item
Algebraic functions in sufficiently high number of variables can not be 
represented as superposition of algebraic functions in sufficiently low (in 
particular, $2$) number of variables -- a suite of deep results due to 
V.I. Arnold and his followers, obtained by interpreting cohomology classes of a
suitable braid group as obstructions to such representation (cf. Napolitano 
\cite{napolit} for references).

\item
There is an $8$-element poset whose clone of monotone maps cannot be generated 
not only by its binary fragment, but by any finite set of maps (due to 
G. Tardos, cf. \cite[\S 11.5]{lau}).
\end{enumerate}

However, as Banach is apparently interested in the case of $X$ finite, and does 
not impose on $X$ any additional structure, such interpretation seems to be unlikely.

If, however, we will understand the superposition in a more 
``operadic-like'', ``multilinear'' fashion, where each variable occurs only
once, the only possible ways to generate the ternary map (\ref{eq-f}) by a 
binary one $*: X \times X \to X$, are:
\advance\value{equation} by1
\begin{equation}\label{eq-L}\tag{\arabic{equation}L}
f(x,y,z) = (x*y)*z ,
\end{equation}
i.e. the same as in the version of the question from the {\L}o\'s paper 
\cite{los} discussed above, and
\begin{equation}\label{eq-R}\tag{\arabic{equation}R}
f(x,y,z) = x*(y*z) .
\end{equation}
This is also in line with Theorem \ref{th-tri}.

Of course, any map representable in the form (\ref{eq-L}) gives rise, via 
permutation of arguments, to a map representable in the form (\ref{eq-R}), and vice versa. Indeed, the equality (\ref{eq-R}) can be rewritten as 
\begin{equation}\label{eq-12}
f^{(13)}(x,y,z) = (x *^{(12)} y) *^{(12)} z ,
\end{equation}
where 
$$
f^\sigma(x_1, \dots, x_n) = f(x_{\sigma(1)}, \dots, x_{\sigma(n)}) ,
$$
and $\sigma$ is a permutation from $S_n$, the symmetric group in $n$ variables.
Banach is concerned with the case of ``commutative'' maps, which, by all 
accounts, are $n$-ary maps $f$ which are stable under any $\sigma \in S_n$: 
$f = f^\sigma$ (usually, such maps are called symmetric). As it follows from (\ref{eq-12}), a commutative map $f$ is representable in the 
form (\ref{eq-L}) if and only if it is representable in the form (\ref{eq-R}).

In this sense, Banach's claim becomes true. In the case of a finite set $X$,
this is obvious from an elementary counting: the number of binary maps is 
$|X|^{|X|^2}$, so the number of ternary maps of the form (\ref{eq-L}) and (\ref{eq-R}) is less than 
$2|X|^{|X|^2}$ (less, because these maps may coincide for different $*$'s), 
while the number of all ternary maps is $|X|^{|X|^3}$.

\medskip

As an aside note, it seems to be an interesting question to estimate more 
exactly the number of different maps of the form (\ref{eq-L}) and (\ref{eq-R})
on an $n$-element set. A computer count produces the following table, where the
second column contains the number of all binary maps on an $n$-element set, 
$T_L(n)$ denotes the number of ternary maps of the form (\ref{eq-L}),
$T_{LR}(n)$ denotes the number of ternary maps both of the form 
(\ref{eq-L}) and (\ref{eq-R}), and $T_{comm}(n)$ denotes the number of 
commutative ternary maps of the form (\ref{eq-L}) (which coincides with the
number of such maps both of the form (\ref{eq-L}) and (\ref{eq-R})) (so, for $n>1$ we have the obvious inequalities 
$T_{comm}(n) < T_L(n) < T_{LR}(n) < 2T_L(n)$):

\medskip

\begin{center}
\begin{tabular}{|l|r|r|r|r|}
\hline
$n$ & $n^{n^2}$     & $T_L(n)$ & $T_{LR}(n)$ & $T_{comm}(n)$ \\ \hline 
$1$ & $1$           & $1$      & $1$         & $1$           \\ \hline 
$2$ & $16$          & $14$     & $21$        & $5$           \\ \hline
$3$ & $19683$       & $19292$  & $38472$     & $48$          \\ \hline
\end{tabular}
\end{center}

\medskip

In the given range, $T_{comm}(n)$ turns out to be equal to $T_{comm}^{comm}(n)$,
the number of commutative ternary maps of the form (\ref{eq-L}) with 
\emph{commutative} $*$. We do not know the general formulas for $T_L(n)$, 
$T_{LR}(n)$, $T_{comm}(n)$, and $T_{comm}^{comm}(n)$, and refrain here from 
making any conjectures, but this seems to be an interesting topic to study 
further\footnote{
A simple Perl program, \texttt{binary-ternary.pl}, which computes these numbers
for small values of $n$, is available as an ancillary file in the arXiv version
of this note. At the time of this writing, the $3$-term sequences for $T_L(n)$ 
and $T_{LR}(n)$ are absent in the Online Encyclopedia of Integer Sequences.
}.

It is easy to manufacture examples of a ternary map $f$ which cannot be
generated, in the sense of (\ref{eq-L}), by any binary operation $*$, thus
explicitly confirming Banach's claim (in examples below, the ternary maps 
are commutative, but they are not representable as superposition, in the sense 
of (\ref{eq-L}), of any binary map, commutative or not).

\begin{example}
Let $X$ be a set containing more than $2$ elements, and $a,b\in X$ be two 
distinct elements of $X$. Define a ternary map $f$ on $X$ by
$$
f(x,y,z) = \begin{cases}
a, & \text{ if all $x,y,z$ are distinct from $a$}          \\
b, & \text{ if at least one of $x,y,z$ coincides with $a$} .
\end{cases}
$$
Suppose (\ref{eq-L}) holds for some binary map $*$ on $X$. If for any two 
elements $x,y\in X$, both distinct from $a$, $x*y \ne a$, then for any $3$
elements $x,y,z\in X$, each distinct from $a$, we have 
$a = f(x,y,z) = (x*y)*z \ne a$, a contradiction. Hence there are $u,v\in X$,
both distinct from $a$, such that $u*v=a$. Then $a*a = (u*v)*a = f(u,v,a) = b$,
and then for any $x \in X$, $b*x = (a*a)*x = f(a,a,x) = b$, and, finally,
$a = f(b,u,v) = (b*u)*v = b*v = b$, a contradiction.
\end{example}

\begin{example}\label{ex-gf}
Let $X$ be a set of $q = p^n$ elements, where $p$ is a prime, and $f$ a ternary
map on $X$. Endow $X$ with the structure of the finite field $\mathsf{GF}(q)$. 
The question is whether there exists or not a binary map $g$ on $\mathsf{GF}(q)$
such that
\begin{equation}\label{eq-fg}
f(x,y,z) = g(g(x,y),z)
\end{equation} 
for any $x,y,z\in \mathsf{GF}(q)$. Since, due to Lagrange interpolation formula,
each $k$-ary function on $\mathsf{GF}(q)$ can be represented as a polynomial in
$k$ variables with coefficients in $\mathsf{GF}(q)$, and the degree $<q$ in each variable (cf., e.g., 
Lidl--Niederreiter \cite[pp.~368--369]{lidl-nider}), and, moreover, two such polynomial maps are
equal if and only if their coefficients are equal, the condition (\ref{eq-fg})
can be rewritten as a system of polynomial equations in (polynomial) coefficients
of $g$. Suitably choosing $f$, one can obtain a system not having solutions in
$\mathsf{GF}(q)$ (in fact, ``most'' of the $f$'s will do, as the system consists
of $q^3$ equations in $q^2$ unknowns). For example, defining a ternary map
$f: \mathsf{GF}(2) \times \mathsf{GF}(2) \times \mathsf{GF}(2) \to 
\mathsf{GF}(2)$ by $f(x,y,z) = xy+xz+yz$, and writing 
$g(x,y) = a + bx + cy + dxy$ for some $a,b,c,d\in \mathsf{GF}(2)$, we arrive
at the system
$$
a + ab = 0, \quad
b^2    = 0, \quad
bc     = 0, \quad
c + ad = 0,  \quad
bd     = 1,  \quad
cd     = 1,  \quad
d^2    = 0
$$
which, evidently, does not have solutions.
\end{example}

Of course, nothing in these examples is specific to $3$ variables, and they can
be easily extended to $n$-ary maps for arbitrary $n$, and, moreover, to 
non-representability in the form $(x * y) \circ z$ for two binary maps $*$ and 
$\circ$, and similar $n$-ary expressions.

\section{Answer to Banach's question}

The following answers the question of Banach interpreted as above -- i.e. about
extensions of arbitrary ternary maps to those having the form (\ref{eq-L}) -- in affirmative. (We deal with arbitrary, not necessary commutative, maps). 

\begin{mytheorem}\label{th-1}
For any (finite) set $X$ and a ternary map $f: X \times X \times X \to X$, there
is a (finite) set $Y \supset X$ and a binary map $*: Y \times Y \to Y$ such 
that 
$$
f(x,y,z) = (x*y)*z
$$
for any $x,y,z \in X$. 
\end{mytheorem}

\begin{proof}[Proof {\rm (Referee)}]
Define $Y$ as a disjoint union $X \cup (X \times X)$, and the binary map $*$ on it as follows:
\begin{align*}
x * y     &= (x,y)      \\ 
(x,y) * z &= f(x,y,z) ,
\end{align*}
where $x,y,z \in X$. The rest, i.e. $x * (y,z)$ and $(x,y) * (z,t)$, is defined
in an arbitrary way.
\end{proof}

This elementary construction has a drawback that it cannot be easily modified to
obtain similar statements in classes of maps satisfying various conditions. For 
example, the so constructed binary map $*$ is, generally, neither commutative, 
nor associative, even if the initial ternary map $f$ is. To remedy this, one may
adopt just a slightly more complicated approach, based on various, related, 
constructions of envelopes of Lie and other triple systems, an idea going back 
to the pioneering paper \cite{jacobson} of Nathan Jacobson (born in Warsaw, but
otherwise not related to Poland) published only a few years later than the 
question was posed. In this way, we get an alternative proof of a positive 
answer to Banach's question in the narrower -- ``associative'' -- version, given
in the papers \cite{los} and \cite{monk-sioson} as mentioned at the beginning of
\S \ref{sec-los}. Our proof differs from both of them and, as we hope, is shorter and simpler.

\begin{mytheorem}[{\L}o\'s, Monk--Sioson]
For any (finite) commutative ternary semigroup $(X,f)$ there is a (finite)
commutative binary semigroup $(Y,*)$ such that $Y \supset X$ and 
$$
f(x,y,z) = (x*y)*z
$$
for any $x,y,z \in X$. 
\end{mytheorem}

\begin{proof}

Let $M$ be the set consisting of maps $X \to X$ of the form 
$$
m_{x,y}: z \mapsto f(x,y,z)
$$ 
for all $x,y\in X$, and $\mathcal M$ the subsemigroup of the semigroup of all 
maps $X \to X$ with respect to composition, generated by $M$. Note that 
commutativity and associativity of $f$ imply 
$$
m_{x,y} \circ m_{u,v} = m_{u,v} \circ m_{x,y}
$$
for any $x,y,u,v \in X$, so $\mathcal M$ is a commutative semigroup.

Define $Y$ as a disjoint union $X \cup \mathcal M$, and the binary map $*$ on it
as follows: 
\begin{align*}
x * y &= m_{x,y} \\
x * g &= g * x = g(x) \\
g * h &= g \circ h
\end{align*}
for $x,y\in X$, $g,h\in \mathcal M$ ($\circ$ denotes the composition of maps).
As $f$ is commutative, $m_{x,y} = m_{y,x}$, and hence $*$ is commutative.  

For any $x,y,z\in X$, we have 
$$
(x * y) * z = m_{x,y}(z) = f(x,y,z) .
$$ 

It remains to check the associativity of $*$. For $3$ terms which all belong to
$X$, the associativity of $*$ follows from the commutativity of $f$. Similarly, for 3 terms which all belong to $\mathcal M$, the associativity of 
$*$ follows from the associativity of $\circ$. If two terms, say, $g$ and $h$, 
belong to $\mathcal M$, and one, say $x$, belongs to $X$, the associativity of 
$*$ follows from the commutativity of $\circ$ in $\mathcal M$:
\begin{gather*}
(g*h)*x = (g \circ h)*x = g(h(x)) = g*(h(x)) = g*(h*x) ,  \\
(g*x)*h = g(x)*h = h(g(x)) = g(h(x)) = g*(h(x)) = g*(x*h) .
\end{gather*}

In the remaining cases, where one term, $g=m_{u,v}$ ($u,v\in X$), belongs to 
$\mathcal M$, and two terms, $x$ and $y$, belong to $X$, assuming additionally 
$z\in X$, and utilizing commutativity and associativity of $f$, we have:
\begin{multline*}
((x*y)*g)(z) = (m_{x,y}*g)(z) = (m_{x,y} \circ g)(z) = f(x,y,g(z)) =
f(x,y,f(u,v,z)) \\ = f(x,f(u,v,y),z) = f(x,g(y),z) = m_{x,g(y)}(z) = 
(x*g(y))(z) = (x*(y*g))(z) 
\end{multline*}
and
\begin{multline*}
((x*g)*y)(z) = (g(x)*y)(z) = m_{g(x),y}(z) = f(g(x),y,z) = f(f(u,v,x),y,z) 
\\= f(x,f(u,v,y),z) \overset{\text{as above}}{=} (x*(g*y))(z) ,
\end{multline*}
what completes the proof.
\end{proof}

One can deal similarly with not necessary commutative ternary semigroups (one 
needs then to consider, instead of $m_{x,y}$, both ``left'' and ``right'' 
multiplications), with ternary groups (thus recovering a part of Post's results \cite{post}), etc.

Some of the statements of this section may be suitably extended, via a 
straightforward iterative procedure, to the maps $f$ of arbitrary arity, but we 
will not venture into this.

\section*{Acknowledgements}

Thanks are due to Martin Goldstern, Kateryna Pavlyk, Jan Wole\'nski, and 
especially an anonymous referee for useful remarks which improved the text. The 
financial support of FAPESP (grant 13/12050-2), of the Regional Authority of 
the Moravian-Silesian Region (grant MSK 44/3316), and of the Ministry of Education and Science of the Republic of Kazakhstan 
(grant 0828/GF4) is gratefully acknowledged.


\begin{thebibliography}{We2}

\bibitem[Al]{alex} W. Alexiewicz, \emph{Andrzej Alexiewicz about Stefan Banach},
in \cite{banach}, 72--81.

\bibitem[Ar]{arnold} V.I. Arnold, 
\emph{Representation of continuous functions of three variables by the
superposition of continuous functions of two variables}, 
Mat. Sbornik \textbf{48} (1959), 3--74 (in Russian); English translation:
Amer. Math. Soc. Transl. \textbf{28} (1963), 61--147.

\bibitem[Ba1]{banach-paper} S. Banach, \emph{Sur un th\'eor\`eme de M. Sierpi\'nski}, Fund. Math. \textbf{25} (1935), 5--6; reprinted in \cite{banach-oeuvr}, Vol. I, 250--251.

\bibitem[Ba2]{banach-oeuvr} \bysame, \emph{Oeuvres, Vol. I, II}, 
PWN, Warszawa, 1967, 1979.

\bibitem[Ba3]{banach} 
\emph{Stefan Banach. Remarkable Life, Brilliant Mathematics} 
(ed. E. Jakimowicz and A. Miranowicz), 3rd ed., Gda\'nsk Univ. Press, 2011.

\bibitem[Be]{bell} E.T. Bell, \emph{The Search for Truth}, Williams \& Wilkins, Baltimore, 1934.

\bibitem[CP]{turowicz} K. Ciesielski and Z. Pogoda, 
\emph{Conversation with Andrzej Turowicz}, 
Math. Intelligencer \textbf{10} (1988), $\mathcal N$4, 13--20.

\bibitem[D]{duda} R. Duda, \emph{Lwowska Szko{\l}a Matematyczna}, 
Wydawnictwo Uniwersytetu Wroc{\l}awskiego, 2007 (in Polish); 
English translation:
\emph{Pearls from a Lost City: The Lvov School of Mathematics}, AMS, 2014.

\bibitem[E]{erdos} P. Erd\"os, \emph{My Scottish Book ``problems''}, in \cite{scott}, 35--43.

\bibitem[F]{feigenbaum} M. Feigenbaum, \emph{Reflection of the Polish masters: an interview with Stan Ulam and Mark Kac}, Los Alamos Science 1982, $\mathcal{N}$6, 54--65; reprinted in J. Stat. Phys. \textbf{39} (1985), 455--476.

\bibitem[G]{gold} M. Goldstern, \emph{A single binary function is enough}, Proceedings of the 81st Workshop on General Algebra (ed. Czermak et al.), Verlag Johannes Heyn, Klagenfurt, Contributions to General Algebra \textbf{20} (2012), 35--37.

\bibitem[GP]{prytula} I. Guran and Ya. Prytula, 
\emph{J\'ozef Schreier['s] ``On finite base in topological groups''}, 
Mat. Studii \textbf{39} (2013), 3--9.

\bibitem[J]{jacobson} N. Jacobson, \emph{Lie and Jordan triple systems}, Amer. J. Math. \textbf{71} (1949), 149--170; reprinted in \emph{Collected Mathematical Papers, Vol. 2}, Birkh\"auser, 1989, 17--38.

\bibitem[KU]{kac-ulam} M. Kac and S. Ulam, \emph{Mathematics and Logic}, Praeger, New York \& London, 1968; reprinted by Dover, 1992.

\bibitem[L]{lau} D. Lau, \emph{Function Algebras on Finite Sets. A Basic Course on Many-Valued Logic and Clone Theory}, Springer, 2006.

\bibitem[LN]{lidl-nider} R. Lidl and H. Niederreiter, \emph{Finite Fields}, 2nd ed., Cambridge Univ. Press, 1997.

\bibitem[{\L}1]{los-superp} J. {\L}o\'s, \emph{Un th\'eor\`eme sur les superpositions des fonctions d\'efinies dans les ensembles arbitraires}, Fund. Math. \textbf{37} (1950), 84--86.

\bibitem[{\L}2]{los} \bysame, \emph{On the extending of models (I)}, Fund. Math. \textbf{42} (1955), 38--54.

\bibitem[MS]{monk-sioson} D. Monk and F.M. Sioson, \emph{$m$-semigroups, semigroups, and function representations}, Fund. Math. \textbf{59} (1966), 233--241.

\bibitem[Mo]{most} A. Mostowski, \emph{L'oeuvre scientifique de Jan {\L}ukasiewicz dans le domaine de la logique math\'ematique}, Fund. Math. \textbf{44} (1957), 1--11.

\bibitem[Mu]{murawski} R. Murawski, \emph{The Philosophy of Mathematics and Logic in the 1920s and 1930s in Poland}, Wydawnictwo Naukowe Uniwersytetu M. Kopernika, Toru\'n, 2011 (in Polish); English translation: Birkh\"auser, 2014.

\bibitem[Nag]{nagel} E. Nagel, \emph{Impressions and appraisals of analytic philosophy in Europe. I, II}, J. Philosophy \textbf{33} (1936), 5--24, 29--53.

\bibitem[Nap]{napolit} F. Napolitano, \emph{Comments to Problem 1972-27}, in
\emph{Arnold's Problems} (ed. V.I. Arnold), 2nd ed., Springer and Phasis, 2005, 283--284.

\bibitem[Pol]{polish} \emph{Polish Logic 1920--1939} (ed. S. McCall), Oxford Univ. Press, 1967.

\bibitem[PS]{polya-szego} G. P\'olya and G. Szeg\"o, \emph{Aufgaben und Lehrs\"atze aus der Analysis I}, Springer, 1925; English translation from the 4th ed.: \emph{Problems and Theorems in Analysis I}, Springer, 1978.

\bibitem[Pos]{post} E.L. Post, \emph{Polyadic groups}, Trans. Amer. Math. Soc. \textbf{48} (1940), 208--350; reprinted in \emph{Solvability, Provability, Definability: The Collected Works of Emil L. Post}, Birkh\"auser, 1994, 106--248.

\bibitem[Sc]{scott} \emph{The Scottish Book. Mathematics from the Scottish Caf\'e} (ed. R.D. Mauldin), Birkh\"auser, 1981.

\bibitem[Si1]{sierp-hilb13} W. Sierpi\'nski, \emph{Remarques sur les fonctions de plusieurs variables r\'eelles}, Prace Matematyczno-Fizyczne \textbf{41} (1934), 171--175; reprinted in \cite{sierp-oeuvre}, 235--238.

\bibitem[Si2]{sierp-banach} \bysame, \emph{Sur le suites infinies de fonctions d\'efinies dans les ensembles quelconques}, Fund. Math. \textbf{24} (1935), 209--212; reprinted in \cite{sierp-oeuvre}, 255--258.

\bibitem[Si3]{sierp} \bysame, \emph{Sur les fonctions de plusieurs variables}, Fund. Math. \textbf{33} (1945), 169--173; reprinted in \cite{sierp-oeuvre}, 434--438.

\bibitem[Si4]{sierp-oeuvre} \bysame, \emph{Oeuvres Choisies, Tome III}, PWN, Warszawa, 1976.

\bibitem[S{\l}]{slup} J. S{\l}upecki, 
\emph{Der volle dreiwertige Aussagenkalk\"ul}, 
Compt. Rend. S\'eances Soc. Sci. Lett. Varsovie, Cl. III 
\textbf{29} (1936--1937), 9--11; English translation in \cite{polish}, 335--337.

\bibitem[St]{steinh} H. Steinhaus, \emph{Stefan Banach}, Studia Math. Seria Specjalna (1963), 7--15; reprinted in \emph{Selected Papers}, PWN, Warszawa, 1985, 878--886; a variant is published in \cite{banach-oeuvr}, Vol. I, 13--22.

\bibitem[Su]{surma} S.J. Surma, \emph{The logical works of Mordchaj Wajsberg}, in \emph{Initiatives in Logic} (ed. J. Srzednicki), M. Nijhoff, Dordrecht, 1987, 101--115.

\bibitem[U1]{ulam-problems} S. Ulam, \emph{Problems in Modern Mathematics}, John Wiley \& Sons, 1960; reprinted, in a slightly extended form, in 1964, and in 
\emph{Sets, Numbers, and Universes. Selected Works}, The MIT Press, 1974, 505--670.

\bibitem[U2]{ulam} \bysame, \emph{Adventures of a Mathematician}, Scribner, New York, 1976; reprinted by Univ. California Press, 1991.

\bibitem[U3]{ulam-analog} \bysame, 
\emph{Analogies between Analogies. The Mathematical Reports of S.M. Ulam and his Los Alamos Collaborators}
(ed. A.R. Bednarek and F. Ulam), Univ. California Press, 1990.

%\bibitem[V]{vitushkin} A.G. Vitushkin, {\rusi 13-ya problema Gil{\cprime}berta i smezhnye voprosy}, Uspekhi Mat. Nauk \textbf{59} (2004), 11--24 (in Russian); English translation: \emph{On Hilbert's thirteenth problem and related questions}, Russ. Math. Surv. \textbf{59} (2004), 11--25.

\bibitem[V]{vitushkin} A.G. Vitushkin, \emph{On Hilbert's thirteenth problem and related questions}, Uspekhi Mat. Nauk \textbf{59} (2004), 11--24 (in Russian); English translation: Russ. Math. Surv. \textbf{59} (2004), 11--25.

\bibitem[Wa]{wajsberg} M. Wajsberg, \emph{Metalogische Beitr\"age}, Wiadomo\'sci Matematyczne \textbf{43} (1937), 131--168; English translation in \cite{polish}, 285--318; reprinted in \emph{Logical Works} (ed. S.J. Surma), Ossolineum, Wroc{\l}aw, 1977, 172--200.

\bibitem[We1]{webb} D.L. Webb, \emph{Generation of any $n$-valued logic by one binary operation}, Proc. Nat. Acad. Sci. USA \textbf{21} (1935), 252--254.

\bibitem[We2]{webb-vars} \bysame, \emph{The algebra of n-valued logic}, Compt. Rend. S\'eances Soc. Sci. Lett. Varsovie, Cl. III \textbf{29} (1936--1937), 153--168.

\bibitem[Wi]{witkiewicz} S.I. Witkiewicz, \emph{Nienasycenie}, Dom Ksi\c{a}\.{z}ki Polskiej Sp\'o{\l}ka Akcyjna, Warszawa, 1930 (in Polish); English translation: \emph{Insatiability}, Univ. Illinois Press, 1977.

\bibitem[Wo1]{wolenski} J. Wole\'nski, \emph{Mathematical logic in Poland 1900--1939: people, circles, institutions, ideas}, Modern Logic \textbf{5} (1995), 363--405; reprinted in \emph{Essays in the History of Logic and Logical Philosophy}, Jagiellonian Univ. Press, Krak\'ow, 1999.

\bibitem[Wo2]{wolenski-lvov} \bysame, \emph{Logic and foundations of Mathematics in Lvov (1900--1939)}, in \emph{Lvov Mathematical School in the Period 1915--1945 as Seen Today} (ed. B. Bojarski, J. {\L}awrynowicz, and Ya.G. Prytula), Banach Center Publ. \textbf{87} (2009), 27--44.

\end{thebibliography}
\end{document}